\documentclass[12pt]{article}
\usepackage[margin = 1in]{geometry}
\usepackage{amsthm,amssymb,amsmath,graphicx,cite, color,mathrsfs,booktabs}
\usepackage{algpseudocode,algorithm,algorithmicx}

\newcommand{\dmax}{\displaystyle\max}

\newcommand{\R}{{\mathbb R}}

\newcommand{\Ren}{\mathcal{R}}

\newcommand{\Grass}{\mathcal{G}}

\renewcommand{\S}{{\mathscr S}}

\newcommand{\D}{{\mathscr D}}

\newcommand{\A}{\mathbf{A}}

\newcommand{\dist}{{\mathrm{dist}}}

\newcommand{\vertii}[1]{{\vert\kern-0.25ex\vert\kern-0.25ex\vert #1     \vert\kern-0.25ex\vert\kern-0.25ex\vert}}
\usepackage{enumitem}

\newcommand{\Diag}{\operatorname{Diag}}

\DeclareMathOperator*{\argmin}{arg\,min}

\newtheorem{theorem}{Theorem}
\newtheorem{proposition}{Proposition}

\newtheorem{corollary}{Corollary}

\def\transp{^{\text{\sf T}}}

\newcommand{\matr}[1]{\begin{bmatrix} #1 \end{bmatrix}}    

\title{Equivalence and invariance of the chi and Hoffman constants of a matrix}
\author{Javier F. Pe\~na\thanks{Tepper School of Business,
Carnegie Mellon University, USA, {\tt jfp@andrew.cmu.edu}} \and Juan C. Vera\thanks{Department of Econometrics and
Operations Research,
Tilburg University, The Netherlands, {\tt j.c.veralizcano@uvt.nl}}
\and Luis F. Zuluaga\thanks{Department of Industrial and Systems Engineering, Lehigh University, USA, {\tt
luis.zuluaga@lehigh.edu}}}

\begin{document}

\maketitle

\begin{abstract}
We show that the following two condition measures of a full column rank matrix $A \in \R^{m\times n}$ are identical:
the chi constant and a {\em signed} Hoffman constant.  This identity is naturally suggested by the evident
invariance of the chi constant under sign changes of the rows of $A$.  We also show that similar equivalence and
invariance properties extend to variants of the chi and Hoffman constants that depend only on the linear subspace
$A(\R^n):=\{Ax: x\in\R^n\} \subseteq \R^m$.  Finally, we show similar identities between the chi constants and
signed versions of Renegar's and Grassmannian condition measures.

\medskip

\medskip


\end{abstract}

\section{Introduction}
\label{sec.intro}

We show a novel equivalence between the following two condition measures of a matrix that play central roles in
numerical linear algebra and in convex optimization: the chi measure~\cite{BenTT90,Diki74,Fors96,Stew89,Todd90} and
the Hoffman constant~\cite{Hoff52,guler1995,KlatT95,WangL14}.  We also show some similar equivalences for some
variants of these constants.

Let $A\in \R^{m\times n}$ be a full column rank matrix.  The chi constant $\chi(A)$ and its variant $\overline
\chi(A)$ arise in the analysis of weighted least squares problems~\cite{BobrV01,Fors96,ForsS01,HougV97}.  In
particular, $\overline \chi(A)$
plays a central role in the
analysis of Vavasis and Ye's interior-point algorithm for linear programming~\cite{MontT03,VavaY96}.
A remarkable feature of Vavasis and Ye's algorithm is its sole dependence on the matrix $A$ defining the primal and
dual constraints.

The Hoffman constant $H(A)$ is associated to Hoffman's Lemma~\cite{Hoff52,guler1995}, a fundamental {\em error bound}
for systems of linear constraints of the form $Ax \le b$.  The Hoffman constant and other similar error bounds are
used to establish the convergence rate of a wide variety of optimization algorithms~\cite{WangL14,
BeckS15,Garb18,GutmP19,LacoJ15,LeveL10,LuoT93,NecoNG18,Pang97,PenaR16,WangL14}.

As we discuss in Section~\ref{sec.def},  the chi constant $\chi(A)$ and its variant $\overline\chi(A)$ can be seen as
measures of {\em worst} behavior of a canonical solution mapping for the following weighted least squares problems
\[ 
b \mapsto \argmin_{x\in \R^n} (Ax-b)\transp D(Ax - b)
\] 
where $D$ is a diagonal matrix with positive diagonal entries.

Similarly, the Hoffman constant $H(A)$ and its variant $\overline H(A)$ can be seen as measures of {\em worst}
behavior of a canonical solution mapping for the following system of linear inequalities
\[
b\mapsto \{x\in\R^n: Ax \le b\}.
\]

It is not immediately obvious that there should be a relationship between the chi and Hoffman constants.  Nonetheless,
it is known that $H(A)\le \chi(A)$ and that $\chi(A)$ can be arbitrarily larger~\cite{HoT02,PenaVZ19}.   Thus an equivalence
between the constants $H(A)$ and $\chi(A)$ appears impossible.   The main goal of  this paper is to show that this
apparent impossibility can be attributed to and rectified via a canonical {\em sign invariance property} of $\chi(A)$
detailed in equation~\eqref{eq.sign.inv} below.  
Namely, the constant $\chi(A)$ does not change when the signs of some of the rows of $A$ are flipped as the solution
mapping~\eqref{eq.wls} satisfies this sign invariance property.  On the other hand, the constant $H(A)$ does
not satisfy this sign invariance property and thus $H(A)$ and $\chi(A)$ cannot be identical. 
 Our main result (Theorem~\ref{thm.main}) shows that  $\chi(A)$ and
$H(A)$ become identical after properly tweaking $H(A)$ to ensure the sign invariance property.  

A similar type of
invariance consideration yields identities between the
variants $\overline \chi(A)$ and $\overline H(A)$.   Our developments can be further extended to obtain analogous
identities between the four measures $\chi(A), \overline \chi(A), H(A), \overline H(A)$ and the following two popular
condition measures for systems of linear inequalities: Renegar's distance to ill-posedness $\Ren(A)$~\cite{Rene95a}
and the Grassmannian condition measure $\Grass(A)$~\cite{AmelB12}.

The above developments are similar in spirit to  results previously derived by Tun\c{c}el~\cite{Tunc99}, by Todd,
Tun\c{c}el, and Ye~\cite{ToddTY01}, and by
Ho and Tun\c{c}el~\cite{HoT02}.  These articles compare various condition measures for linear programming including
the chi and Hoffman constants.  However, there are two major differences between our developments and theirs.  First,
most of the results in~\cite{Tunc99,HoT02,ToddTY01} provide only inequalities and hence are weaker than our identities
concerning the chi and Hoffman constants.  Second, the articles~\cite{Tunc99,HoT02,ToddTY01} do not deal with
Renegar's and Grassmannian condition measures but instead relate the chi and Hoffman constants with Ye's condition
measure~\cite{Ye94} for polyhedra of the form $\{A\transp y: y\ge 0, \|y\|_1 = 1\}$.  Hence we deliberately chose not
to discuss Ye's condition measure in this paper.  However, we note that our results can be extended to
identities involving Ye's condition measure by drawing on the recent work by Pe\~na and Roshchina~\cite{PenaR20}.


To formally state the sign invariance property, we rely on the following convenient notation.  Let $\S \subseteq
\R^{m\times m}$ denote the set of {\em signature} matrices defined as follows
\begin{equation}\label{eq.signature}
\S:=\{\Diag(s): s \in \{-1,1\}^m\}.
\end{equation}
The constant $\chi(A)$ satisfies the following {\em sign invariance property:}
\begin{equation}\label{eq.sign.inv}
\chi(A) = \chi(SA) \text{ for all } S\in \S.
\end{equation}
Our main result states that $\chi(A)$ and $H(A)$ become identical if we take a suitable {\em closure} of $H(A)$ to
ensure the sign invariance property.
\begin{theorem}\label{thm.main}
Let $A\in \R^{m\times n}$ be a full column-rank matrix.  Then
\begin{equation}\label{eq.thm}
\chi(A) = \max_{S\in\S} H(SA).
\end{equation}
\end{theorem}
A similar type of invariance property relates the measures $\chi(A)$ and $\overline \chi(A)$.  The construction of
$\overline \chi(A)$ depends only on the subspace $A(\R^n)$.  Thus $\overline \chi(A)$ readily
satisfies the following invariance  under right multiplication by non-singular matrices
\[
\overline\chi(A) = \overline\chi(AR) \text{ for all } R\in \R^{n\times n} \text{ non-singular}.
\]
In analogy to Theorem~\ref{thm.main}, the measures $\chi(A)$ and $\overline\chi(A)$ become identical if we take a
suitable closure of $\chi(A)$ to ensure the same invariance under right multiplication by non-singular matrices (see
Proposition~\ref{prop.chi.invar}):
\[
\overline \chi(A) = \min_{R \in \R^{m\times m} \atop \text{ non-singular}} \|AR\| \cdot \chi(AR).
\]
Furthermore, the same kind of identity holds for the measures $H(A)$ and $\overline H(A)$ (see
Proposition~\ref{prop.H.invar}):
\[
\overline H(A) = \min_{R \in \R^{m\times m} \atop \text{ non-singular}} \|AR\| \cdot H(AR).
\]
In particular, identity~\eqref{eq.thm} in Theorem~\ref{thm.main} readily extends to the measures $\overline \chi(A)$
and $\overline H(A)$ as follows (see Corollary~\ref{cor}):
\begin{equation}\label{eq.thm.bar}
\overline \chi(A) = \max_{S\in\S} \overline H(SA).
\end{equation}
Our proof of Theorem~\ref{thm.main} will actually show the following stronger identity when all rows of $A$ are
non-zero (see Theorem~\ref{thm.main.strong}):
\[
\chi(A) = \max_{S\in\S\atop SA(\R^n) \cap \R^m_{++} \ne \emptyset} H(SA).
\]
This stronger identity in turn yields some interesting connections with Renegar's distance to ill-posedness
$\Ren(A)$~\cite{Rene95a,Rene95b} and the Grassmannian condition number of $\Grass(A)$~\cite{AmelB12}.  More precisely,
in Section~\ref{sec.renegar} we show the following identity analogous to~\eqref{eq.thm} (see
Proposition~\ref{prop.chi.ren}):
\begin{equation}\label{eq.chi.ren.0}
\chi(A) = \max_{S\in\S\atop SA(\R^n) \cap \R^m_{++} \ne \emptyset} \frac{1}{\Ren(SA)}
\end{equation}
and the following identity analogous to~\eqref{eq.thm.bar}  (see Corollary~\ref{corol.chi.grass}):
\begin{equation}\label{eq.chi.grass.0}
\overline\chi(A) = \max_{S\in\S\atop SA(\R^n) \cap \R^m_{++} \ne \emptyset}\Grass(SA).
\end{equation}

The main sections of the paper are organized as follows.  Section~\ref{sec.def} recalls the construction of the chi
constants $\chi(A), \overline \chi(A)$ as well as the Hoffman constants $H(A), \overline H(A)$ and some of their main
properties.  Our presentation deliberately follows separate but similar formats for $\chi(A), \overline \chi(A)$ and
for $H(A), \overline H(A)$.  Section~\ref{sec.proof} presents the proof of Theorem~\ref{thm.main}.  To do so, we state
and prove the stronger Theorem~\ref{thm.main.strong}.  Finally, Section~\ref{sec.renegar} recalls the construction of
Renegar's condition measure $\Ren(A)$ and of the Grassmannian condition measure $\Grass(A)$.  This section also
proves identities~\eqref{eq.chi.ren.0} and~\eqref{eq.chi.grass.0}.

Throughout the paper whenever we encounter an Euclidean space $\R^d$ we implicitly assume that it is endowed with the
Euclidean norm defined by the canonical inner product in $\R^d$, that is, $\|u\|:=\sqrt{u\transp u}$ for all $u\in
\R^d$.  Likewise, whenever we encounter a space of matrices $\R^{p\times d}$ we implicitly assume that it is endowed
with the operator norm, that is,
\[
\|A\| = \max_{x\in \R^d\atop \|x\|\le 1} \|Ax\|
\]
for all $A\in \R^{p\times d}$.

\section{Definition and properties of the chi and Hoffman constants}
\label{sec.def}

This section recalls the construction and main properties of the constants $\chi(A), \overline \chi(A)$ and
$H(A),\overline H(A)$.  These constants can be seen as condition measures for two fundamental problems in scientific
computing, namely {\em weighted least squares} and {\em linear inequalities.}

\subsection{Weighted least squares}

Let $\D\subseteq \R^{m\times m}$ denote the set of diagonal matrices in $\R^{m\times m}$ with positive diagonal
entries.  That is,
\[
\D:=\{\Diag(d): d\in\R^m_{++}\},
\]
where $\R^m_{++} \subseteq \R^m$ denotes the set of vectors in $\R^m$ with positive entries.

Suppose $A\in \R^{m\times n}$. Given $D \in \D$, consider the weighted least squares problem
\begin{equation}\label{eq.wls}
\min_{x\in \R^n} \; (Ax - b)\transp D (Ax-b).
\end{equation}
When $A$ is full column-rank, it is easy to see that the solution to~\eqref{eq.wls} is precisely $A_D^{\dagger}b$
where $A_D^\dagger$ is the following weighted pseudo-inverse of $A$~\cite{Fors96,Stew89}:
\begin{equation}\label{eq.wpseudo}
A_D^{\dagger} = (A\transp D A)^{-1} A\transp D.
\end{equation}
\subsubsection{Condition measures $\chi(A)$ and $\overline \chi(A)$}
Suppose $A\in \R^{m\times n}$ is full column-rank. The condition measure $\chi(A)$ is defined as the following
worst-case characteristic of the family of solution mappings $A^{\dagger}_D: \R^m\rightarrow \R^n$ constructed
via~\eqref{eq.wpseudo}:
\begin{equation}\label{eq.def.chi}
\chi(A):=\max_{D\in\D} \|A^{\dagger}_D\|.
\end{equation}
Consider the following alternative reformulation of the weighted least-squares problem~\eqref{eq.wls} in the subspace
$A(\R^n)$:
\begin{equation}\label{eq.wls.image}
\min_{y\in A(\R^n)} \; (y - b)\transp D (y-b).
\end{equation}
The solution to~\eqref{eq.wls.image} is evidently the $D$-projection of $b$ onto $A(\R^n)$.  Once again, it is easy to
see that if $A$ is full column-rank then the $D$-projection  onto $A(\R^n)$ is
\[
A(A\transp D A)^{-1} A\transp D = AA_D^{\dagger}.
\]
The condition measure $\overline\chi(A)$ is defined as the
following worst-case characteristic of the family of solution mappings
$AA_D^{\dagger}: \R^m\rightarrow A(\R^n)$:
\begin{equation}\label{eq.def.barchi}
\overline\chi(A):=\max_{D\in\D} \|AA_D^{\dagger}\| = \max_{D\in\D} \|A(A\transp D A)^{-1} A\transp D\|.
\end{equation}
Although it is not immediately evident, the constants $\chi(A)$ and $\overline \chi(A)$ are finite for any full-rank matrix $A\in
\R^{m\times n}$.  This fact was independently shown by Ben-Tal and Teboulle~\cite{BenTT90},
Dikin~\cite{Diki74}, Stewart~\cite{Stew89}, and Todd~\cite{Todd90}.  The constants $\chi(A)$ and $\overline \chi(A)$
arise in and
play a key role in weighted least-squares problems~\cite{Fors96,ForsS01,BobrV01} and in linear
programming~\cite{HoT02,ToddTY01,Tunc99,VavaY96}.

We record some alternative expressions for $\chi(A)$ and $\overline \chi(A)$ that are closely related to the
constructions of $H(A)$ and $\overline H(A)$ discussed below.  First, observe that
\[
\chi(A)= \max_{D\in \D} \max_{b\in \R^m, x\in \R^n\atop Ax \ne b} \frac{\|x-A^{\dagger}_D(b)\|}{\|Ax-b\|}.
\]
Second, observe that
\[
\overline\chi(A)= \max_{D\in \D} \max_{b\in \R^m, y\in A(\R^n)\atop y \ne b} \frac{\|y-AA_D^{\dagger}(b)\|}{\|y-b\|}.
\]
\subsubsection{Properties of $\chi(A)$ and $\overline \chi(A)$}

Suppose $A\in \R^{m\times n}$ is full column-rank and $D\in \D$.  By construction, the solution mappings
$A_D^{\dagger}$ and $AA_D^{\dagger}$ satisfy the following property: For  $S\in \S$ then
$(SA)_D^\dagger = A_D^\dagger S$.  In particular $\|(SA)_D^\dagger\| =\|A_D^{\dagger}\|$ and $\|(SA)(SA)_D^\dagger\| = \|AA_D^\dagger\|$.  Therefore~\eqref{eq.def.chi} and~\eqref{eq.def.barchi} imply that the
constants $\chi(A)$ and $\overline \chi(A)$ satisfy the following sign invariance property: 
\[
\chi(A) = \chi(SA) \text{ and } \overline\chi(A) = \overline\chi(SA), \text{ for all }S \in \S.
\]
Furthermore, the quantity $\overline\chi(A)$ depends only on the subspace $A(\R^n)$ which evidently satisfies
$A(\R^n) = AR(\R^n)$ for all non-singular $R\in \R^{n\times n}$.
Therefore, the constant $\overline \chi(A)$ is invariant under multiplication by non-singular matrices, that is, 
\begin{equation}\label{eq.chi.right.invar}
\overline \chi(A) = \overline \chi(AR), \text{ for all non-singular }R\in \R^{n\times
n}.
\end{equation}
The constant $ \chi(A)$ is not invariant under multiplication by singular matrices. Proposition~\ref{prop.chi.invar} shows that $\overline\chi(A)$ is the closure of $ \chi(A)$ under this kind of invariance.

\begin{proposition}\label{prop.chi.invar}
Suppose $A\in \R^{m\times n}$ is full column-rank.  Then $\overline \chi(A) \le \|A\| \cdot \chi(A)$ and $\overline
\chi(A) = \chi(A)$ when the columns of $A$ are orthonormal.  In particular,
\begin{equation}\label{eq.chi.invar}
\overline \chi(A) = \min_{R \in \R^{n\times n} \atop \text{ non-singular}} \|AR\| \cdot \chi(AR).
\end{equation}
\end{proposition}
\begin{proof}
Since $\|AA_D^\dagger\| \le \|A\| \cdot \|A_D^\dagger\|$, the construction~\eqref{eq.def.chi}
and~\eqref{eq.def.barchi} of $\chi(A)$ and $\overline \chi(A)$
readily implies that
\begin{equation}\label{eq.chi.sub}
\overline \chi(A) \le \|A\| \cdot \chi(A).
\end{equation}
Next, we show that $\overline \chi(A) = \chi(A)$ when the columns of $A$ are orthonormal.  To that end, observe that
if the columns of $A$ are orthonormal then $\|Ax\| = \|x\|$ for all $x\in \R^n$.  In particular, if the columns of $A$
are orthonormal then $\|AA_D^\dagger\| = \|A_D^\dagger\|$ for all $D\in\D$.  Thus~\eqref{eq.def.chi}
and~\eqref{eq.def.barchi} imply that $\overline \chi(A) = \chi(A)$.

Finally, from~\eqref{eq.chi.right.invar} and~\eqref{eq.chi.sub} it follows that
$
\overline \chi(A)  = \overline \chi(AR) \le \|AR\|\cdot \chi(AR)
$
for all $R\in\R^{m\times m}$ non-singular.  Thus~\eqref{eq.chi.invar} follows.

\end{proof}
In the special case when $m=n$ and  $A\in \R^{n\times n}$ is non-singular it is easy to see that
\[
\chi(A) = \|A^{-1}\|.
\]
We will rely on the following related characterization of $\chi(A)$
from~\cite{Fors96}.  The same characterization is also stated and proved in~\cite{Zhan00} by adapting a technique
from~\cite{ToddTY01}.  In the statement below for $A\in \R^{m\times n}$ and $J\subseteq[m] := \{1,\dots,m\}$ the matrix $A_J\in \R^{J \times n}$ denotes the $|J|\times n$ submatrix of $A$ defined by the rows of $A$ indexed by $J$.

\begin{proposition}\label{prop.chi} Suppose $A\in \R^{m\times n}$ has full column-rank.  Then
\begin{equation}\label{eq.chi}
\chi(A) = \max_{J\subseteq[m]\atop A_J \text{non-singular}} \|A_J^{-1}\|  =
\max_{J\subseteq[m]\atop A_J \text{non-singular}} \max_{v \in \R^J\atop \|A_J\transp v \|= 1} \|v\|.
\end{equation}
\end{proposition}

\subsection{Linear inequalities}

Suppose $A\in \R^{m\times n}$. Consider the feasibility problem
\begin{equation}\label{eq.li}
Ax \le b.
\end{equation}
The solution of~\eqref{eq.li} is the set
\begin{equation}\label{eq.poly}
P_A(b):=\{x\in \R^n: Ax \le b\}.
\end{equation}
Observe that $P_A(b) \ne \emptyset$ if and only if $b \in A(\R^n) + \R^m_+$.

\subsubsection{Condition measures $H(A)$ and $\overline H(A)$}

Suppose $A\in \R^{m\times n}$ is a nonzero matrix. The condition measure $H(A)$ is defined as the following worst-case
characteristic of the solution mapping $P_A:\R^m \rightrightarrows \R^n$ constructed via~\eqref{eq.poly}:
\begin{equation}\label{eq.def.H}
H(A) = \max_{b \in A(\R^n) + \R^m_+\atop x  \in \R^n\setminus P_A(b)} \frac{\dist(x,P_A(b))}{\|(Ax-b)_+\|}.
\end{equation}
Here and throughout the paper, $\dist(u,S)$ denotes the following point-to-set distance for all $u\in \R^d$ and
$S\subseteq \R^d$:
\[
\dist(u,S) = \inf_{v\in S} \|u-v\|.
\]
The constant $H(A)$ can be equivalently defined as the smallest constant depending only on $A$ such that the following
{\em error bound} holds for all $b \in A(\R^n)+\R^m_+$ and all $x\in \R^n$:
\[
\dist(x,P_{A}(b)) \le H(A) \cdot \|(Ax-b)_+\|.
\]
Again, it is not immediately evident that $H(A)$ is finite.  This fact was shown by Hoffman in his seminal
paper~\cite{Hoff52}.  Other proofs of this fundamental result can be found in~\cite{guler1995,PenaVZ19,WangL14}.
After Hoffman's initial work, the literature in error bounds has developed
extensively~\cite{luo1994,LuoT93,ngai2015,Nguy17,Pang97,ZhouS17}.  Error bounds play a key role in optimization and
variational analysis.  In particular, error bounds are widely used to established the convergence rate of a variety of
algorithms~\cite{BeckS15,Garb18,GutmP19,LacoJ15,LeveL10,LuoT93,NecoNG18,Pang97,PenaR16,WangL14}.

Consider the following reformulation of~\eqref{eq.li} in the subspace $A(\R^n)$:
\begin{equation}\label{eq.li.image}
y \le b, \; y \in A(\R^n).
\end{equation}
The solution of~\eqref{eq.li.image} is the set
\[
(b-\R^m_+) \cap A(\R^n) = AP_A(b).
\]
Define $\overline H(A)$ as the
following  worst-case characteristic of the solution mapping $AP_A:\R^m\rightrightarrows A(\R^n)$:
\begin{equation}\label{eq.def.barH}
\overline H(A) = \max_{b \in A(\R^n)+\R^m_+ \atop y \in A(\R^n) \setminus AP_A(b)}
\frac{\dist(y,AP_A(b))}{\|(y-b)_+\|}.
\end{equation}
The constant $\overline H(A)$ can be equivalently defined as the smallest constant depending only on the subspace
$A(\R^n)$ such that the following error bound holds for all $b\in A(\R^n) + \R^m_+$ and $v\in A(\R^n) + b$
\[
\dist(v,(A(\R^n) + b)\cap\R^m_+) \le \overline H(A) \cdot \dist(v,\R^m_+).
\]
\subsubsection{Properties of $H(A)$ and $\overline H(A)$}

By construction, $\overline H(A)$ depends only on  $A(\R^n)$ and thus is invariant under multiplication by non-singular matrices, i.e., 
\begin{equation}\label{eq.H.right.invar}
\overline H(A) = \overline H(AR), \text{ for all  non-singular }  R\in \R^{n\times n}.
\end{equation}
On the other hand, $H(A)$ is not invariant under multiplication by non-singular matrices.  Proposition~\ref{prop.H.invar} shows that  $\overline H(A)$ is the closure of $H(A)$ under this kind of invariance.
\begin{proposition}\label{prop.H.invar}
Suppose $A\in \R^{m\times n}$ is a nonzero matrix.  Then $\overline H(A) \le \|A\| \cdot H(A)$ and $\overline H(A) =
H(A)$ when the \underline{nonzero} columns of $A$ are orthonormal.  In particular,
\begin{equation}\label{eq.H.invar}
\overline H(A) = \min_{R \in \R^{n\times n} \atop \text{ non-singular}} \|AR\| \cdot H(AR).
\end{equation}
\end{proposition}
\begin{proof}
This proof is similar to the proof of Proposition~\ref{prop.chi.invar}.  Observe that  $\dist(Ax,AP_A(b)) \le
\|A\|\cdot \dist(x,P_A(b))$ for all $x\in \R^n$ because $\|Ax - Au\| \le \|A\|\cdot \|x-u\|$ for all $x,u\in \R^n$.
Hence~\eqref{eq.def.H} and~\eqref{eq.def.barH} imply that
\begin{equation}\label{eq.H.sub}
\overline H(A) \le \|A\| \cdot H(A).
\end{equation}
We next show that $\overline H(A) = H(A)$ when the nonzero columns of $A$ are orthonormal.  For ease of exposition,
consider first the case when all columns of $A$ are nonzero and orthonormal.  In this case it is easy to see that $y
\in  A(\R^n)$ if and only if $y = Ax$ for some unique $x\in \R^n$ with $\|y\| = \|x\|$.
Therefore $\dist(y,AP_A(b)) = \dist(x,P_A(b))$ for all $y = Ax\in A(\R^n)$.
From~\eqref{eq.def.H} and~\eqref{eq.def.barH} it follows that $\overline H(A) = H(A)$.

Next consider the more  general case when some columns of $A$ are zero.  Without loss of generality assume that $A =
\matr{\tilde A & 0}$ for some $\tilde A\in \R^{m \times k}$ with nonzero orthonormal columns for some $k < n$.  Since
the columns of $\tilde A$ are orthonormal, we have $\overline H(\tilde A) = H(\tilde A)$.  To finish, it suffices to
show that $\overline H( A) = \overline H(\tilde A)$ and
$ H( A) =  H(\tilde A)$.  Indeed, $\overline H( A) = \overline H(\tilde A)$ holds because $A(\R^n) = \tilde A(\R^k)$
and $AP_A(b) = \tilde A P_{\tilde A}(b)$.  On the other hand, for  $x\in \R^n$ let $\tilde x\in \R^k$ denote the
subvector of first $k$ entries of $x$.  Then $Ax = \tilde A \tilde x$ for all $x\in \R^n$ and thus $P_A(b) = P_{\tilde
A}(b) \times \R^{n-k}$.  Hence
\[
H(A) = \max_{b \in A(\R^n) + \R^m_+\atop x  \in \R^n\setminus P_A(b)} \frac{\dist(x,P_A(b))}{\|(Ax-b)_+\|} =
\max_{b \in \tilde A(\R^k) + \R^m_+\atop \tilde x  \in \R^n\setminus P_A(b)} \frac{\dist(\tilde x,P_{\tilde
A}(b))}{\|(\tilde A\tilde x-b)_+\|} = H(\tilde A).
\]
Finally from~\eqref{eq.H.right.invar}  and~\eqref{eq.H.sub} it follows that
$
\overline H(A)  = \overline H(AR) \le \|AR\|\cdot H(AR)
$
for all $R\in\R^{m\times m}$ non-singular.  Thus~\eqref{eq.H.invar} follows.

\end{proof}

We will also rely on the following two properties of $H(A)$. First, in the special case when $A(\R^n) \cap \R^m_{++}
\ne \emptyset$ or equivalently $A(\R^n) + \R^m_+ = \R^m$ we have~\cite[Corollary 1]{PenaVZ19}
\begin{equation}\label{eq.H.simple}
H(A) = \max_{v\in \R^m_+\atop \|A\transp v\| =1} \|v\|.
\end{equation}
Second, for general $A\in \R^{m\times n}$ we have the following related characterization of $H(A)$ discussed
in~\cite{PenaVZ19} but that can be traced back to~\cite{KlatT95,WangL14,Zhan00}.

\begin{proposition}\label{prop.H} Suppose $A\in \R^{m\times n}$ is full column-rank.  Then
\begin{equation}\label{eq.H}
H(A) = \max_{J\subseteq[m]\atop A_J(\R^n) \cap \R^J_{++} \ne \emptyset} \max_{v \in \R^J_+ \atop \|A_J\transp v\| = 1}
\|v\| = \max_{J\subseteq[m]\atop A_J \text{non-singular}} \max_{v \in \R^J_+ \atop \|A_J\transp v\| = 1} \|v\|.
\end{equation}
\end{proposition}

Observe both the similarity and subtle difference between the right-most expressions in the
characterization~\eqref{eq.chi} of $\chi(A)$ in Proposition~\ref{prop.chi} and the characterization~\eqref{eq.H} of
$H(A)$ in Proposition~\ref{prop.H}: the first maximum is taken over the same collection of sets $J$ in
both~\eqref{eq.chi} and~\eqref{eq.H}  whereas the second maximum is taken over $v\in\R^J$ in~\eqref{eq.chi} and over
$v\in \R^J_+$ in~\eqref{eq.H}.

\section{Proof of Theorem~\ref{thm.main}}
\label{sec.proof}
We will prove the following stronger version of Theorem~\ref{thm.main}.

\begin{theorem}\label{thm.main.strong}
Let $A\in \R^{m\times n}$ be a full column-rank matrix.  Then
\begin{equation}\label{eq.thm.full}
\chi(A) = \max_{S\in\S} H(SA) = H(\A),
\end{equation}
where
$\A\in\R^{2m \times n}$ is the column-wise concatenation of $A$ and $-A$, that is,
\begin{equation}\label{eq.concatenate}
\A = \matr{A \\ -A}.
\end{equation}
Furthermore, if all rows of $A$ are nonzero then~\eqref{eq.thm.full} can be sharpened to
\begin{equation}\label{eq.thm.sharper}
\chi(A) = \max_{S\in\S \atop SA(\R^n) \cap \R^m_{++} \ne \emptyset} H(SA).
\end{equation}
\end{theorem}
\begin{proof}
From~\eqref{eq.chi} in Proposition~\ref{prop.chi} and~\eqref{eq.H} in Proposition~\ref{prop.H} it immediately follows
that $H(A) \le \chi(A)$.  Thus the sign invariance of $\chi(A)$ readily yields
\[
\chi(A) = \max_{S\in\S} \chi(SA) \ge \max_{S\in\S} H(SA).
\]
To prove the reverse inequality we rely on~\eqref{eq.chi} and~\eqref{eq.H} again.  Suppose $\hat J \subseteq [m]$ is
such that $A_{\hat J}$ is non-singular and
\[
\chi(A) = \|A_{\hat J}^{-1}\| = \max_{v\in \R^{\hat J}\atop \|A_{\hat J}\transp v\|=1} \|v\|.
\]
Thus $\chi(A) = \|\hat v\|$ for some $\hat v \in \R^{\hat J}$ such that $\|A_{\hat J}\transp \hat v\|=1$.  Choose
$\hat S\in \S$ such that $\hat S_{ii} = \text{sign}(v_i)$ for each $i\in \hat J$ and let $u:=\hat S_{\hat J}\hat v \in
\R^{\hat J}_+$.  Observe that $(\hat SA)_J = \hat S_{\hat J} A_{\hat J}$ is nonsingular and
\[
\|(\hat SA)_{\hat J}\transp u\| = \|A_{\hat J}\transp \hat S_{\hat J} u \|= \|A_{\hat J}\transp v\| = 1.
\]
Therefore
\[
\max_{S\in\S} H(SA) \ge H(\hat SA) \ge \max_{w \in \R^{\hat J}_+ \atop \|(\hat SA)_{\hat J}\transp w\|=1} \|w\| \ge
\|u\| = \|\hat v\| = \chi(A).
\]
Thus the first identity in~\eqref{eq.thm.full} is established.  Next, Proposition~\ref{prop.chi} and
Proposition~\ref{prop.H} imply that for all $S\in \S$
\[
\chi(A) = \chi(\A) \ge H(\A) \ge H(SA).
\]
The second inequality follows because all rows of $SA$ are rows of $\A$ as well.  Hence by taking the maximum over
$S\in\S$ and applying the first identity in~\eqref{eq.thm.full}, we obtain the second identity in~\eqref{eq.thm.full}.

When all rows of $A$ are non-zero, it follows that $A \tilde v$ has all nonzero entries for an arbitrarily small
perturbation $\tilde v$ of $\hat v$.  Therefore the matrix $\hat S \in \S$ above can be chosen so that both $\hat
S_{\hat J}\hat v \in \R^{\hat J}_+$ and $\hat S A\transp \tilde v \in \R^m_{++}.$  Thus the sharper
identity~\eqref{eq.thm.sharper} follows.
\end{proof}

\begin{corollary}\label{cor}
Let $A\in \R^{m\times n}$ be a full column-rank matrix.  Then
\[
\overline\chi(A) = \max_{S\in\S} \overline H(SA) = \overline H(\A),
\]
where $\A$ is as in~\eqref{eq.concatenate}.
Furthermore, if all rows of $A$ are nonzero then
\[
\overline\chi(A) = \max_{S\in\S \atop SA(\R^n) \cap \R^m_{++} \ne \emptyset} \overline H(SA).
\]
\end{corollary}
\begin{proof}
This is an immediate consequence of Theorem~\ref{thm.main.strong}, Proposition~\ref{prop.chi.invar} and
Proposition~\ref{prop.H.invar}.
\end{proof}

We note that when $A\in \R^{m\times n}$ is full column-rank but some rows of $A\in \R^{m\times n}$ are zero, then the
following amended version of~\eqref{eq.thm.sharper} holds for the submatrix $\tilde A\in \R^{\ell \times n}$ obtained
after deleting the zero rows from $A$:
\[
\chi(\tilde A) = \max_{S\in\S \atop S\tilde A(\R^n) \cap \R^\ell_{++} \ne \emptyset} H(S\tilde A).
\]
The construction of $\chi(A)$ and $H(A)$ enables us to rewrite  the latter identity as follows
\[
\chi(A) = \max_{S\in\S \atop S\tilde A(\R^n) \cap \R^\ell_{++} \ne \emptyset} H(SA).
\]

\section{Renegar's and Grassmannian condition numbers}
\label{sec.renegar}

Suppose $A \in \R^{m\times n}$ is such that $A(\R^n)\cap \R^m_{++} \ne \emptyset$.  This property can be equivalently
stated as $A(\R^n) + \R^m_{+} = \R^m$, that is, for all $b\in \R^m$ the system of linear inequalities
\[
Ax \le b
\]
is feasible.  In his seminal paper on condition measures for optimization~\cite{Rene95a}, Renegar
defined the {\em distance to infeasibility} of $A$ as the smallest perturbation that can be made on $A$ so that this
property is lost.  That is
\[
\Ren(A):=\inf\{\|\Delta A\|: (A+\Delta A)(\R^n)\cap \R^m_{++} = \emptyset\}.
\]
Renegar also defined $\|A\|/\Ren(A)$ as a condition number of $A$.

We have the following characterization of $\chi(A)$ in terms $\Ren(A)$ analogous to that in
Theorem~\ref{thm.main.strong}.
\begin{proposition}\label{prop.chi.ren}
Let $A\in \R^{m\times n}$ be a full column-rank matrix.  If $A(\R^n)\cap \R^m_{++}\ne \emptyset$ then $H(A) = 1/\Ren(A)$.  Consequently, if all rows of full column-rank matrix $A\in\R^{m\times n}$ are nonzero then
\begin{equation}\label{eq.chi.ren}
\chi(A) = \max_{S\in\S \atop SA(\R^n) \cap \R^m_{++} \ne \emptyset} \frac{1}{\Ren(SA)}.
\end{equation}
\end{proposition}
\begin{proof}
When $A(\R^n)\cap \R^m_{++} \ne \emptyset$, the distance to ill-posedness $\Ren(A)$ has the following property similar
in spirit to Proposition~\ref{prop.chi} and Proposition~\ref{prop.H} (see\cite[Theorem 3.5]{Rene95b}):
\begin{equation}\label{eq.Ren}
\frac{1}{\Ren(A)} = \dmax_{v\in \R^m_+ \atop \|v\| =1} \|A\transp v\|.
\end{equation}
From~\eqref{eq.H.simple} and~\eqref{eq.Ren} it follows that $H(A) = 1/\Ren(A)$ when $A(\R^n)\cap \R^m_{++} \ne
\emptyset$.  The latter condition and~\eqref{eq.thm.sharper} in turn imply \eqref{eq.chi.ren} if all rows of $A$ are nonzero.
\end{proof}

Ameluxen and Burgisser~\cite{AmelB12} proposed a condition number via the Grassmannian manifold of linear subspaces of
$\R^m$ of some fixed dimension.  This condition number can be seen as a variant of Renegar's condition measure that
depends only on $A(\R^n)$ akin to the variants $\overline \chi(A)$ and $\overline H(A)$ of $\chi(A)$ and $H(A)$
respectively.  We next recall the description of the Grassmannian condition number proposed by Ameluxen and
Burgisser~\cite{AmelB12}. First,  define the {\em Grassmannian} distance $\dist(L,L')$ between two linear subspaces
$L,L' \subseteq \R^m$ of the same dimension  as
\[
\dist(L,L'):=\|\Pi_L - \Pi_{L'}\|,
\]
where $\Pi_L$ and $\Pi_{L'}$ denote the orthogonal projection matrices onto $L$ and $L'$ respectively.

Suppose $A\in \R^{m\times n}$  satisfies $A(\R^n)\cap \R^m_{++} \ne\emptyset$.  Let $L:=A(\R^n)$ and define the {\em
Grassmannian condition number} of $A$ as follows
\[
\Grass(A):=\frac{1}{\min\{\dist(L,L'): \dim(L') = \dim(L) \text{ and } L\cap \R^m_{++} = \emptyset\}}.
\]

Since $\Grass(A)$ depends only on $A(\R^n)$, it automatically satisfies the following invariance property just as
$\overline \chi(A)$ and $\overline H(A)$ do: For all non-singular $R\in \R^{m\times m}$
\begin{equation}\label{eq.grass.invar}
\Grass(AR) =   \Grass(A).
\end{equation}
The pair of quantities $1/\Ren(A), \Grass(A)$ are related to each other in the same way the pairs of quantities $\chi(A),\overline\chi(A)$ and $H(A),\overline H(A)$ are.  More precisely, we have the following analogue of Proposition~\ref{prop.chi.invar} and Proposition~\ref{prop.H.invar}. 

\begin{proposition}\label{prop.ren.grass.invar}
Suppose $A\in \R^{m\times n}$ is a nonzero matrix and $A(\R^n) \cap \R^m_{++}\ne \emptyset$. Then $\Grass(A) \le \|A\|/\Ren(A)$ and $\Grass(A) = 1/\Ren(A)$ when the non-zero columns of $A$ are orthonormal.  Consequently, if $A\in \R^{m\times n}$ is a nonzero matrix
\begin{equation}\label{eq.ren.grass.invar}
\Grass(A) = \min_{R \in \R^{m\times m} \atop \text{non-singular}} \frac{\|AR\|}{\Ren(AR)}.
\end{equation}
\end{proposition}
\begin{proof}
Suppose $A\in \R^{m\times n}$ and $A(\R^n) \cap \R^m_{++}\ne \emptyset$.  Then
the inequality $\Grass(A) \le 1/\Ren(A)$ follows from~\cite[Theorem 1.4]{AmelB12} and the identity $\Grass(A) = 1/\Ren(A)$ when the nonzero columns of $A$ are orthonormal follows from~\cite[Theorem 1.3]{AmelB12}.  The latter two facts and~\eqref{eq.grass.invar}
in turn imply~\eqref{eq.ren.grass.invar} when $A\in \R^{m\times n}$ is a nonzero matrix.
\end{proof}

We conclude with the following characterization of $\overline \chi(A)$ in terms $\Grass(A)$ analogous to that
in Corollary~\ref{cor}.  
\begin{corollary}\label{corol.chi.grass}
Suppose $A\in \R^{m\times n}$ is a full column-rank matrix and all rows of $A$ are nonzero. Then
\begin{equation}\label{eq.chi.grass}
\overline\chi(A) = \max_{S\in\S \atop SA(\R^n) \cap \R^m_{++} \ne \emptyset} \Grass(SA).
\end{equation}
\end{corollary}
\begin{proof}
This is an immediate consequence of  Proposition~\ref{prop.chi.invar}, Proposition~\ref{prop.chi.ren}, and
Proposition~\ref{prop.ren.grass.invar}.
\end{proof}

\bibliographystyle{plain}


\begin{thebibliography}{10}

\bibitem{AmelB12}
D.~Amelunxen and P.~B\"urgisser.
\newblock A coordinate-free condition number for convex programming.
\newblock {\em SIAM J. on Optim.}, 22(3):1029--1041, 2012.

\bibitem{BeckS15}
A.~Beck and S.~Shtern.
\newblock Linearly convergent away-step conditional gradient for non-strongly
  convex functions.
\newblock {\em Mathematical Programming}, 164:1--27, 2017.

\bibitem{BenTT90}
A.~Ben-Tal and M.~Teboulle.
\newblock A geometric property of the least squares solution of linear
  equations.
\newblock {\em Linear Algebra and its Applications}, 139:165--170, 1990.

\bibitem{BobrV01}
E.~Bobrovnikova and S.~Vavasis.
\newblock Accurate solution of weighted least squares by iterative methods.
\newblock {\em SIAM Journal on Matrix Analysis and Applications},
  22(4):1153--1174, 2001.

\bibitem{Diki74}
I.~Dikin.
\newblock On the speed of an iterative process.
\newblock {\em Upravlyaemye Sistemi}, 12(1):54--60, 1974.

\bibitem{Fors96}
A.~Forsgren.
\newblock On linear least-squares problems with diagonally dominant weight
  matrices.
\newblock {\em SIAM Journal on Matrix Analysis and Applications},
  17(4):763--788, 1996.

\bibitem{ForsS01}
A.~Forsgren and G.~Sporre.
\newblock On weighted linear least-squares problems related to interior methods
  for convex quadratic programming.
\newblock {\em SIAM Journal on Matrix Analysis and Applications}, 23(1):42--56,
  2001.

\bibitem{Garb18}
D.~Garber.
\newblock Fast rates for online gradient descent without strong convexity via
  {H}offman's bound.
\newblock {\em arXiv preprint arXiv:1802.04623}, 2018.

\bibitem{guler1995}
O.~G{\"u}ler, A.~Hoffman, and U.~Rothblum.
\newblock Approximations to solutions to systems of linear inequalities.
\newblock {\em SIAM Journal on Matrix Analysis and Applications},
  16(2):688--696, 1995.

\bibitem{GutmP19}
D.~Gutman and J.~Pe{\~n}a.
\newblock The condition number of a function relative to a set.
\newblock {\em To Appear in Math. Program.}, 2020.

\bibitem{HoT02}
J.~Ho and L.~Tun\c{c}el.
\newblock Reconciliation of various complexity and condition measures for
  linear programming problems and a generalization of {T}ardos' theorem.
\newblock In {\em Foundations Of Computational Mathematics}, pages 93--147.
  World Scientific, 2002.

\bibitem{Hoff52}
A.~Hoffman.
\newblock On approximate solutions of systems of linear inequalities.
\newblock {\em Journal of Research of the National Bureau of Standards},
  49(4):263--265, 1952.

\bibitem{HougV97}
P.~Hough and S.~Vavasis.
\newblock Complete orthogonal decomposition for weighted least squares.
\newblock {\em SIAM Journal on Matrix Analysis and Applications},
  18(2):369--392, 1997.

\bibitem{KlatT95}
D.~Klatte and G.~Thiere.
\newblock Error bounds for solutions of linear equations and inequalities.
\newblock {\em Zeitschrift f{\"u}r Operations Research}, 41(2):191--214, 1995.

\bibitem{LacoJ15}
S.~Lacoste-Julien and M.~Jaggi.
\newblock On the global linear convergence of {F}rank-{W}olfe optimization
  variants.
\newblock In {\em Advances in Neural Information Processing Systems (NIPS)},
  2015.

\bibitem{LeveL10}
D.~Leventhal and A.~Lewis.
\newblock Randomized methods for linear constraints: Convergence rates and
  conditioning.
\newblock {\em Math. Oper. Res.}, 35:641--654, 2010.

\bibitem{luo1994}
X. Luo and Z. Luo.
\newblock Extension of {H}offman's error bound to polynomial systems.
\newblock {\em SIAM Journal on Optimization}, 4(2):383--392, 1994.

\bibitem{LuoT93}
Z.~Luo and P.~Tseng.
\newblock Error bounds and convergence analysis of feasible descent methods: a
  general approach.
\newblock {\em Annals of Operations Research}, 46(1):157--178, 1993.

\bibitem{MontT03}
R.~Monteiro and T.~Tsuchiya.
\newblock A variant of the {V}avasis-{Y}e layered-step interior-point algorithm
  for linear programming.
\newblock {\em SIAM Journal on Optimization}, 13(4):1054--1079, 2003.

\bibitem{NecoNG18}
I.~Necoara, Y.~Nesterov, and F.~Glineur.
\newblock Linear convergence of first order methods for non-strongly convex
  optimization.
\newblock {\em Math. Program.}, 175:69--107, 2019.

\bibitem{ngai2015}
H.~Ngai.
\newblock Global error bounds for systems of convex polynomials over polyhedral
  constraints.
\newblock {\em SIAM Journal on Optimization}, 25(1):521--539, 2015.

\bibitem{Nguy17}
T.~Nguyen.
\newblock A stroll in the jungle of error bounds.
\newblock {\em arXiv preprint arXiv:1704.06938}, 2017.

\bibitem{Pang97}
J.~S. Pang.
\newblock Error bounds in mathematical programming.
\newblock {\em Math. Program.}, 79:299--332, 1997.

\bibitem{PenaR16}
J.~Pe{\~n}a and D.~Rodr\'iguez.
\newblock Polytope conditioning and linear convergence of the {F}rank-{W}olfe
  algorithm.
\newblock {\em Mathematics of Operations Research}, 44(1):1--18, 2019.

\bibitem{PenaR20}
J.~Pe{\~n}a and V.~Roshchina.
\newblock A data-independent distance to infeasibility for linear conic
  systems.
\newblock {\em SIAM J. on Optim.}, 30(2):1049--1066, 2020.

\bibitem{PenaVZ19}
J.~Pe{\~n}a, J.~Vera, and L.~Zuluaga.
\newblock New characterizations of {H}offman constants for systems of linear
  constraints.
\newblock {\em To Appear in Math. Program.}, 2020.

\bibitem{Rene95a}
J.~Renegar.
\newblock Incorporating condition measures into the complexity theory of linear
  programming.
\newblock {\em SIAM J. on Optim.}, 5:506--524, 1995.

\bibitem{Rene95b}
J.~Renegar.
\newblock Linear programming, complexity theory and elementary functional
  analysis.
\newblock {\em Math. Program.}, 70:279--351, 1995.

\bibitem{Stew89}
G.~Stewart.
\newblock On scaled projections and pseudoinverses.
\newblock {\em Linear Algebra and its Applications}, 112:189--193, 1989.

\bibitem{Todd90}
M.~Todd.
\newblock A {D}antzig-{W}olfe-like variant of {K}armarkar's interior-point
  linear programming algorithm.
\newblock {\em Operations Research}, 38(6):1006--1018, 1990.

\bibitem{ToddTY01}
M.~Todd, L.~Tun{\c{c}}el, and Y.~Ye.
\newblock Characterizations, bounds, and probabilistic analysis of two
  complexity measures for linear programming problems.
\newblock {\em Mathematical Programming}, 90(1):59--69, 2001.

\bibitem{Tunc99}
L.~Tun{\c{c}}el.
\newblock On the condition numbers for polyhedra in {K}armarkar's form.
\newblock {\em Operations Research Letters}, 24(4):149--155, 1999.

\bibitem{VavaY96}
S.~Vavasis and Y.~Ye.
\newblock A primal-dual interior point method whose running time depends only
  on the constraint matrix.
\newblock {\em Mathematical Programming}, 74(1):79--120, 1996.

\bibitem{WangL14}
P.~Wang and C.~Lin.
\newblock Iteration complexity of feasible descent methods for convex
  optimization.
\newblock {\em Journal of Machine Learning Research}, 15(1):1523--1548, 2014.

\bibitem{Ye94}
Y.~Ye.
\newblock Toward probabilistic analysis of interior-point algorithms for linear
  programming.
\newblock {\em Math. of Oper. Res.}, 19:38--52, 1994.

\bibitem{Zhan00}
S.~Zhang.
\newblock Global error bounds for convex conic problems.
\newblock {\em SIAM Journal on Optimization}, 10(3):836--851, 2000.

\bibitem{ZhouS17}
Z.~Zhou and A.~So.
\newblock A unified approach to error bounds for structured convex optimization
  problems.
\newblock {\em Mathematical Programming}, 165(2):689--728, 2017.

\end{thebibliography}
\end{document}